\documentclass[10pt, reqno]{amsart}
\usepackage{amsmath} 
\usepackage{amsthm}
\usepackage{amssymb}
\usepackage{amsfonts}
\usepackage{mathdots}
\usepackage{latexsym}
\usepackage{graphicx} 
\usepackage{setspace} 
\usepackage{hyperref}
\usepackage{cite} 
\usepackage{multicol} 
\usepackage[table]{xcolor}

\usepackage[sc]{mathpazo}
\usepackage[T1]{fontenc}

\newcommand{\ceil}[1]{\lceil #1 \rceil}

\newcommand{\minimatrix}[4]{\begin{bmatrix} #1 & #2 \\ #3 & #4 \end{bmatrix}  }

\theoremstyle{plain}

\newtheorem{Lemma}{Lemma}
\theoremstyle{definition}

\allowdisplaybreaks

\begin{document}

    \title{On the matrix equation $XA+AX^T =0$}

    \author{Stephan Ramon Garcia}
    \address{   Department of Mathematics\\
            Pomona College\\
            Claremont, California\\
            91711 \\ USA}
    \email{Stephan.Garcia@pomona.edu}
    \urladdr{\url{http://pages.pomona.edu/~sg064747}}

    \author{Amy L. Shoemaker}

    \keywords{Congruence, congruence canonical form, Lie algebra, Pascal's triangle, Pascal matrix.}
    \subjclass[2010]{15A24, 15A21}

    \thanks{Partially supported by National Science Foundation Grant DMS-1001614.}

    \begin{abstract}
	The matrix equation $XA + AX^T = 0$, which has relevance to the study of Lie algebras, was recently studied by 
	De Ter\'an and Dopico (Linear Algebra Appl. {\bf 434}~(2011), 44--67).  They reduced the study of 
	this equation to several special cases and produced explicit solutions in most instances.
	In this note we obtain an explicit solution in one of the difficult cases,
	for which only the dimension of the solution space and an algorithm to find a basis of this space were known previously.
    \end{abstract}

\maketitle

\section{Introduction}

	In the recent article \cite{DTD}, F.~De Ter\'an and F.M.~Dopico considered the equation
	\begin{equation}\label{eq-DTD}
		XA+AX^T = 0,
	\end{equation}
	where $A$ is a given square complex matrix.  
	For each fixed $A$, the solution space to \eqref{eq-DTD} is 
	a Lie algebra $\mathfrak{g}(A)$ equipped with Lie bracket $[X,Y]:=XY - YX$.
	To characterize, up to similarity, all possible Lie algebras $\mathfrak{g}(A)$ which can arise in this
	manner, it turns out that one need only consider 
	those $A$ which are in \emph{congruence canonical form}
	(recall that two matrices $A$ and $B$ are \emph{congruent} if $A = P^T BP$ for some nonsingular $P$) \cite[Sect.~2]{DTD}.

	To be more specific, it is known that each square complex matrix is congruent to a direct sum, 
	uniquely determined up to a permutation of summands, 
	of three types of canonical matrices (see \cite[Thm.~2.1]{HSCM} for complete details).	
	In general, the Lie algebras corresponding to these canonical matrices are far from semisimple and many of them display quite striking
	visual patterns (e.g., \cite[p.~52, p.~54]{DTD}).  
	
	We are concerned here with canonical matrices of Type II.  These are the matrices
	\begin{equation*}
		H_{2n}(\mu):= \minimatrix{0}{I_n}{J_n(\mu)}{0},
	\end{equation*}
	where $I_n$ denotes the $n \times n$ identity matrix and $J_n(\mu)$ is the $n \times n$ Jordan matrix
	whose eigenvalue $\mu$ satisfies $0 \neq \mu \neq (-1)^{n+1}$. 
	De Ter\'an and Dopico were unable to find an explicit description of the Lie algebra $\mathfrak{g}(H_{2n}((-1)^n))$,
	although they computed the dimension of this algebra for even $n$ using an algorithmic approach
	and stated the corresponding results for odd $n$ without proof \cite[Appendix A]{DTD}.

	We aim here to provide an explicit solution to the matrix equation
	\begin{equation}\label{eq-Main}	
		X H_{2n}(\mu)  +  H_{2n}(\mu) X^T=0.
	\end{equation}
	In particular, our method does not significantly depend upon the parity of $n$.	
	Besides providing an explicit description of the Lie algebra $\mathfrak{g}(H_{2n}((-1)^n))$,
	our approach is also notable for its use of a family of matrices derived from Pascal's Triangle.

\section{Solution}

	We begin our consideration of the matrix equation \eqref{eq-Main} by partitioning
	$X$ conformally with the decomposition of $H_{2n}(\mu)$:
	\begin{equation*}
		X = \minimatrix{A}{B}{C}{D}.
	\end{equation*}
	Substituting the preceding into \eqref{eq-Main} yields the system of equations
	\begin{align}
		 B J_n(\mu) +B^T &=0, \label{eq-B} \\
		A+ D^T &= 0,  \label{eq-AD1}\\
		D J_n(\mu)  + J_n(\mu) A^T &=0, \label{eq-AD2}\\
		C+ J_n(\mu)C^T&=0,  \label{eq-C}
	\end{align}
	which together are equivalent to \eqref{eq-Main}.  Before proceeding,
	let us note that solving for $X$ in \eqref{eq-Main}, taking the transpose of the resulting expression,
	and substituting the final result back into \eqref{eq-Main} shows that
	$X$ commutes with the matrix
	\begin{equation}\label{eq-Cosquare}
		H_{2n}(\mu) H^{-T}_{2n}(\mu)  = J_n(\mu)  \oplus  J_n^{-T}(\mu).
	\end{equation}	
	As noted in \cite[Lem.~6]{DTD}, these observations
	are sufficient to handle the case $\mu \neq \pm 1$.  In this case,
	the matrices $J_n^{-T}(\mu)$ and $J_n(\mu)$
	have distinct eigenvalues whence $B = C = 0$ by \cite[Cor.~9.1.2]{GLR}.
	Since $D$ commutes with $J_n(\mu)$ by \eqref{eq-AD1} and \eqref{eq-AD2}, 
	it follows from \cite[Thm.~9.1.1]{GLR} that $D$ is an upper-triangular Toeplitz matrix
	and that $X = (-D^T) \oplus D$.

	Now suppose that $\mu = (-1)^n$.  As before, we see that
	$D$ is an upper-triangular Toeplitz matrix and that $A = -D^T$.  The entire difference
	between our two cases lies in the structure of the off-diagonal blocks $B$ and $C$.

	We first focus our attention upon the matrix $B$.
	Taking the transpose of \eqref{eq-B} and substituting the result back 
	into \eqref{eq-B} yields 
	\begin{equation}\label{eq-CJKMKT}
		B J_n(\mu) =  J_n^{-T}(\mu) B.
	\end{equation}
	Since $\mu = (-1)^n$, it turns out that the Jordan canonical form of $J_n^{-T}(\mu)$ is
	precisely $J_n(\mu)$.  In fact, the similarity is implemented by a special matrix
	derived from Pascal's Triangle.
		
	\begin{Lemma}\label{LemmaSimilar}
		We have 
		\begin{equation}\label{eq-Similarity}
			J_n^{-T}(\mu) = \Delta_n J_n(\mu) \Delta_n^{-1},
		\end{equation}	
		where $\Delta_n$ is the $n \times n$ matrix whose entries are given by
		\begin{equation}\label{eq-DeltaEntries}
			[\Delta_n]_{i,j} 
			 = \mu^i(-\mu)^{j-1} \binom{j-1}{n-i}.
		\end{equation}
		In particular, we observe the convention that $\binom{n}{k}=0$ whenever $k > n$.
		Thus for $n=1,3,5,\ldots$ we have
		\begin{equation*}
			\Delta_1 = [1],\qquad
			\Delta_3 = 
			\begin{bmatrix}
				 0 & 0 & -1 \\
				 0 & 1 & 2 \\
				 -1 & -1 & -1 \\
			\end{bmatrix}
			\qquad
			\Delta_5 = \small
			\begin{bmatrix}
				 0 & 0 & 0 & 0 & -1 \\
				 0 & 0 & 0 & 1 & 4 \\
				 0 & 0 & -1 & -3 & -6 \\
				 0 & 1 & 2 & 3 & 4 \\
				 -1 & -1 & -1 & -1 & -1 \\
			\end{bmatrix},
		\end{equation*}
		and so forth.   For $n=2,4,6,\ldots$ we have
		\begin{equation*}	
			\Delta_2 =	
			\minimatrix{0}{-1}{1}{-1},\quad
			\Delta_4 = 
			\begin{bmatrix}
				 0 & 0 & 0 & -1 \\
				 0 & 0 & 1 & -3 \\
				 0 & -1 & 2 & -3 \\
				 1 & -1 & 1 & -1 \\
			\end{bmatrix},
			\quad
			\Delta_6 = \small
			\begin{bmatrix}
				 0 & 0 & 0 & 0 & 0 & -1 \\
				 0 & 0 & 0 & 0 & 1 & -5 \\
				 0 & 0 & 0 & -1 & 4 & -10 \\
				 0 & 0 & 1 & -3 & 6 & -10 \\
				 0 & -1 & 2 & -3 & 4 & -5 \\
				 1 & -1 & 1 & -1 & 1 & -1 \\
			\end{bmatrix},
		\end{equation*}
		and so on.
	\end{Lemma}

	\begin{proof}[Pf.of Lemma \ref{LemmaSimilar}]
		We first note that \eqref{eq-Similarity} is equivalent to 
		\begin{equation}\label{eq-FTS}
			\Delta_n = J_n^T(\mu) \Delta_n J_n(\mu).
		\end{equation}
		Upon writing 
		\begin{equation*}
			J_n(\mu) = \mu I_n + J_n(0),
		\end{equation*}
		we find that \eqref{eq-FTS} is equivalent to
		\begin{equation}\label{eq-FUS}
			J_n^T(0) \Delta_n J_n(0)+ \mu  \Delta_n J_n(0) + \mu J_n^T(0) \Delta_n =0.
		\end{equation}
		Using Pascal's Rule and the fact that $\mu = (-1)^n$, the $(i,j)$ entry in the right hand side of \eqref{eq-FUS} is
		\begin{align}
			&[\Delta_n]_{i-1,j-1} + \mu [ \Delta_n]_{i,j-1} + \mu [\Delta_n]_{i-1,j} \nonumber\\
			&\quad= \mu^{i-1} (-\mu)^j \binom{j-2}{n-i+1} + \mu^{i-1} (-\mu)^j \binom{j-2}{n-i} + \mu^i (-\mu)^{j-1} \binom{j-1}{n-i+1} \nonumber\\
			&\quad = \mu^{i-1}(-\mu)^j \left[ \binom{j-2}{n-i+1} + \binom{j-2}{n-i} - \binom{j-1}{n-i+1} \right] \label{eq-CRXP}\\
			&\quad = 0.\nonumber
		\end{align}
		For $i+j < n$, we have used the fact that all of the binomial coefficients appearing in \eqref{eq-CRXP} vanish.
	\end{proof}

	\begin{Lemma}\label{LemmaInverse}
		The inverse of $\Delta_n$ is the reflection of $\Delta_n$ with respect to its center.  In other words,
		\begin{equation*}
			[\Delta_n^{-1}]_{i,j} = [\Delta]_{n-i+1,n-j+1}
		\end{equation*}
		for $1 \leq i,j \leq n$. 
	\end{Lemma}
	
	\begin{proof}[Pf.~of Lemma \ref{LemmaInverse}]
		Let $\widetilde{\Delta}_n$ denote the reflection of $\Delta_n$ with respect to its center.
		First note that $[\widetilde{\Delta}_n \Delta_n]_{i,j} = 0$ holds whenever $i>j$, due to the ``triangular''
		structure of the matrices involved.  On the other hand, if
		$j \geq i$, then it follows that
		\begin{align*}
			[ \widetilde{\Delta}_n \Delta_n]_{i,j}
			&= \sum_{k=1}^n [\widetilde{\Delta}_n]_{i,k} [\Delta]_{k,j}\\
			&= \sum_{k=1}^n [\Delta]_{n-i+1,n-k+1} [\Delta]_{k,j} \\
			&= \sum_{k=1}^n \mu^{n-i+1}(-\mu)^{n-k} \binom{n-k}{i-1}  \mu^k(-\mu)^{j-1} \binom{j-1}{n-k} \\
			&= \mu^{i+j}\sum_{k=1}^n (-1)^{n-k-(j-1)}\binom{n-k}{i-1}   \binom{j-1}{n-k} \\
			&= \mu^{i+j}\sum_{\ell=i-1}^{j-1} (-1)^{\ell-(j-1)}\binom{\ell}{i-1}   \binom{j-1}{\ell} \\
			&= \mu^{i+j}\sum_{\ell=i-1}^{j-1} (-1)^{\ell-(j-1)}\binom{j-1}{i-1}   \binom{j-i}{\ell-i+1} \\
			&= \mu^{i+j} \binom{j-1}{i-1}\sum_{\ell=i-1}^{j-1} (-1)^{\ell-(j-1)}   \binom{j-i}{\ell-i+1} \\
			&= \mu^{i+j} (-1)^{i+j} \binom{j-1}{i-1} \sum_{r=0}^{j-i} (-1)^r    \binom{j-i}{r} \\
			&= \mu^{i+j} (-1)^{i+j} \binom{j-1}{i-1} (1-1)^{j-i}\\
			&= \delta_{i,j}.
		\end{align*}
		Thus $\widetilde{\Delta}_n \Delta_n = I_n$, as claimed.
	\end{proof}

	\begin{Lemma}\label{LemmaIdentity}
		The identity
		\begin{equation}\label{eq-IDS}
			\big(\Delta_n J_n^{\ell}(\mu) \big)^T = - \Delta_n J_n^{n-\ell}(\mu)
		\end{equation}
		holds for $\ell = 0,1,\ldots,n$.  	
	\end{Lemma}
	
	\begin{proof}[Pf.~of Lemma \ref{LemmaIdentity}]
		We first claim that it suffices to prove that 
		\begin{equation}\label{eq-Basic}
			\Delta_n^T = -\Delta_n J_n^n(\mu)
		\end{equation}
		holds.  Indeed, using \eqref{eq-Similarity} and \eqref{eq-Basic} it follows that
		\begin{align*}
			\Delta_n^T 
			&= -\Delta_n J_n^n(\mu) \\
			&= -(\Delta_n J_n^{\ell}(\mu)\Delta_n^{-1}) \Delta_n J_n^{n-\ell}(\mu) \\
			&= -\big(J_n^{\ell}(\mu)\big)^{-T} \Delta_n J_n^{n-\ell}(\mu),
		\end{align*}
		which implies \eqref{eq-IDS}.
	
		In order to prove \eqref{eq-Basic},
		it suffices to derive the equivalent identity
		\begin{equation*}
			\Delta^{-1}_n \Delta_n^T = - J_n^n(\mu).
		\end{equation*}
		Computing the $(i,j)$ entry in the matrix product $\Delta^{-1}_n \Delta_n^T$, we see that
		\begin{align}
			[\Delta_n^{-1} \Delta_n^T]_{i,j}
			&= \sum_{k=1}^n [ \Delta_n^{-1}]_{i,k} [\Delta_n^T]_{k,j}  \nonumber\\
			&= \sum_{k=1}^n [ \Delta_n]_{{n-i+1},{n-k+1}} [\Delta_n]_{j,k}  \nonumber\\
			&= \sum_{k=1}^n \mu^{n-i+1}(-\mu)^{n-k} \binom{n-k}{i-1} \mu^j (-\mu)^{k-1} \binom{k-1}{n-j}  \nonumber\\
			&= -\mu^{i+j+1}\sum_{k=1}^n \binom{n-k}{i-1} \binom{k-1}{n-j}  \nonumber\\
			&=- \mu^{i+j+1} \sum_{\ell=0}^{n-1} \binom{n-\ell-1}{i-1} \binom{\ell}{n-j}  \label{eq-FR1}\\
			&=- \mu^{i+j+1} \binom{n}{n+i-j} \label{eq-FR2}\\
			&=- \mu^{i+j+1} \binom{n}{j-i} \nonumber\\
			&=-[J_n^n(\mu)]_{i,j}, \nonumber
		\end{align}
		where the passage from \eqref{eq-FR1} to \eqref{eq-FR2}
		follows by computing the coefficient of $x^{n-1}$ in the identity
		\begin{equation*}
			\frac{x^{i-1}}{(1-x)^i} \cdot \frac{ x^{n-j} }{ (1-x)^{n-j+1} } =
			\frac{x^{n+i-j-1}}{ (1-x)^{n+i-j+1}}.\qedhere
		\end{equation*}
	\end{proof}

	Having digressed somewhat upon the properties of the matrix $\Delta_n$, let us return now 
	to the equation \eqref{eq-CJKMKT}.  Upon substituting
	\eqref{eq-Similarity} into \eqref{eq-CJKMKT}, we find that
	\begin{equation*}
		(\Delta_n^{-1}B)J_n(\mu) = J_n(\mu)(\Delta_n^{-1}B),
	\end{equation*}
	from which it follows that $B = \Delta_n U$ for some upper triangular 
	Toeplitz matrix $U$ \cite[Thm.~9.1.1]{GLR}.  However, it is important to note that 
	not every such $U$ leads to a solution of the original equation \eqref{eq-B}.
		
	Now observe that the matrices $I_n, J_n(\mu),\ldots, J_n^{n-1}(\mu)$ form a basis for the space
	of $n \times n$ upper triangular Toeplitz matrices.  To see this, simply note that $J_n^k(\mu)$ contains 
	nonzero entries on the $k$th superdiagonal whereas the $k$th superdiagonals of the matrices 
	$I_n,J_n(\mu), J_n^2(\mu),\ldots, J_n^{k-1}(\mu)$ are identically zero.
	Since $U$ is an upper triangular Toeplitz matrix, we may write
	\begin{equation}\label{eq-UTJ}
		U = \sum_{j=0}^{n-1} \alpha_{j} J_n^{j}(\mu)
	\end{equation}
	for some constants $\alpha_0,\alpha_1,\ldots,\alpha_{n-1}$.

	Substituting $B = \Delta_n U$ into \eqref{eq-B} yields
	\begin{equation}\label{eq-DDKFU}
		\Delta_n U = - \big(\Delta_nU\big)^T J_n^{-1}(\mu),
	\end{equation}
	which, in light of \eqref{eq-UTJ}, provides us with
	\begin{equation*}
		\sum_{j=0}^{n-1} \alpha_{j} \Delta_n J_n^{j}(\mu)
		= -\sum_{k=0}^{n-1} \alpha_{k} \big( \Delta_n J_n^{k}(\mu) \big)^T J_n^{-1}(\mu).
	\end{equation*}
	Using the identity \eqref{eq-IDS} in the preceding formula, we obtain
	\begin{equation*}
		\sum_{j=0}^{n-1} \alpha_{j}  J_n^{j}(\mu)
		=  \sum_{k=0}^{n-1} \alpha_{k} J_n^{n-k-1}(\mu),
	\end{equation*}
	which can be rewritten as
	\begin{equation*}	
		\sum_{\ell=0}^{n-1} (\alpha_{\ell} - \alpha_{n-1-\ell})  J_n^{\ell}(\mu) = 0.
	\end{equation*}
	Since the matrices $I_n,J_n(\mu), J_n^2(\mu),\ldots, J_n^{n-1}(\mu)$ are linearly independent, it follows that
	the matrix \eqref{eq-UTJ} is determined by exactly $\ceil{\frac{n}{2}}$ free parameters.
	Thus the general solution to \eqref{eq-B} is
	\begin{equation*}
		\boxed{ B = \Delta_n \!\!\!\sum_{\ell=0}^{ \ceil{ \frac{n}{2}}-1} 
		\!\!\alpha_{\ell} \left[ J_n^{\ell}\big( (-1)^n\big) + J_n^{n-1-\ell}\big((-1)^n\big) \right]. }
	\end{equation*}
	
	For $n = 3$, the general form of the matrix $B$ is
	\begin{equation*}
		\begin{bmatrix}
			 0 & 0 & 2 \left(\alpha _1-\alpha _0\right) \\
			 0 & 2 \left(\alpha _0-\alpha _1\right) & 2 \left(\alpha _0-\alpha _1\right) \\
			2 \left(\alpha _1-\alpha _0\right) & 0 & -\alpha _0
		\end{bmatrix}
		=
		\begin{bmatrix}
			 0 & 0 & -\beta _1 \\
			 0 & \beta _1 & \beta _1 \\
			 -\beta _1 & 0 & -\beta _0
		\end{bmatrix},
	\end{equation*}
	where $\beta_0 = \alpha_0$ and $\beta_1 = 2(\alpha_0 - \alpha_1)$.  For $n = 4$ we have	
	\begin{equation*}
		\begin{bmatrix}
			 0 & 0 & 0 & -2 \alpha _0-2 \alpha _1 \\
			 0 & 0 & 2 \alpha _0+2 \alpha _1 & -3 \alpha _0-3 \alpha _1 \\
			 0 & -2 \alpha _0-2 \alpha _1 & \alpha _0+\alpha _1 & -3 \alpha _0-\alpha _1 \\
			 2 \alpha _0+2 \alpha _1 & \alpha _0+\alpha _1 & 2 \alpha _0 & -\alpha _0
		\end{bmatrix},
	\end{equation*}
	and for $n = 5$ we obtain
	\begin{equation*}\footnotesize
		\begin{bmatrix}
			 0 & 0 & 0 & 0 & -2 \alpha _0+2 \alpha _1-2 \alpha _2 \\
			 0 & 0 & 0 & 2 \alpha _0-2 \alpha _1+2 \alpha _2 & 4 \alpha _0-4 \alpha _1+4 \alpha _2 \\
			 0 & 0 & -2 \alpha _0+2 \alpha _1-2 \alpha _2 & -2 \alpha _0+2 \alpha _1-2 \alpha _2 & -6 \alpha _0+3 \alpha _1-2 \alpha _2 \\
			 0 & 2 \alpha _0-2 \alpha _1+2 \alpha _2 & 0 & 4 \alpha _0-\alpha _1 & 4 \alpha _0-\alpha _1 \\
			 -2 \alpha _0+2 \alpha _1-2 \alpha _2 & 2 \alpha _0-2 \alpha _1+2 \alpha _2 & -4 \alpha _0+\alpha _1 & 0 & -\alpha _0
		\end{bmatrix}.
	\end{equation*}
	It is typographically impractical to depict solutions for $n \geq 6$

	To complete our solution to the original matrix equation \eqref{eq-Main}, we follow \cite[p.~55]{DTD} 
	in noting that $B$ satisfies \eqref{eq-B} 
	if and only if the matrix $C = RBR$, where $R$ denotes
	the reversed identity matrix, satisfies the corresponding equation \eqref{eq-C}.	
	Thus the dimension of the solution space to our original equation
	\eqref{eq-Main} is $n + 2 \ceil{ \frac{n}{2} }$ whenever $\mu = (-1)^n$.

\smallskip
\noindent\textbf{Acknowledgments}:  We wish to thank the anonymous referee, who provided a large number of
helpful comments and suggestions.  We also thank Andrii Dmytryshyn for pointing out the references
\cite{Dmytryshyn} and \cite{Dmytryshyn2} to us.

\bibliographystyle{amsplain}
\bibliography{OMEXAAX}

\end{document}